\documentclass[SEDP]{cedram-sem}

\usepackage{amsfonts}
\usepackage{amsthm}
\usepackage[dvipsnames,svgnames,table]{xcolor}
\usepackage[colorlinks=true, pdfstartview=FitV,linkcolor=ForestGreen,citecolor=ForestGreen, urlcolor=blue]{hyperref}

\theoremstyle{plain}

\newtheorem{theorem}{Theorem}
\newtheorem{proposition}[theorem]{Proposition}
\newtheorem{lemma}[theorem]{Lemma}

\theoremstyle{definition}

\newtheorem{example}[theorem]{Example}

\newtheorem{remark}[theorem]{Remark}%
\newcommand{\R}{{\mathbb R}}
\newcommand{\cQ}{\mathcal{Q}}
\newcommand{\cB}{\mathcal{B}}

\newcommand{\dd} {\; \mathrm{d}}
\newcommand{\dsigma}{\dd \sigma}
\newcommand{\dt}{\dd t}
\newcommand{\dr}{\dd r}
\newcommand{\dz}{\dd z}
\newcommand{\dw}{\dd w}
\newcommand{\dv}{\dd v}

\newcommand{\fin}{f_{\mathrm{in}}}
\newcommand{\Lambdaloc}{\Lambda_{\Delta}}
\newcommand{\bhs}{{b_{\mathrm{hs}}}}

\title[Fisher information for the Boltzmann equation]{Fisher information for solutions of the Boltzmann equation}

\alttitle{L'information de Fisher des solutions de l'équation de Boltzmann}

\author[C. \lastname{Imbert}]{\firstname{Cyril} \lastname{Imbert}}

\address{Département de mathématiques et applications\\
\'Ecole normale supérieure, Université PSL, CNRS\\
75005 Paris, France}
\thanks{Cyril Imbert thanks the Simons Laufer Mathematical Sciences Institute in Berkeley, California, where
he was in residence when this note was written (NSF Grant No. DMS-2424139).}
\email{cyril.imbert@math.cnrs.fr}

  
\altkeywords{L'équation de Boltzmann, information de Fisher, inégalité de log-Sobolev}


\begin{document}

\begin{abstract}
  This note reviews a recent contribution about the Fisher information for the space-homogeneous Boltzmann equation by L.~Silvestre, C.~Villani and the author (\href{https://arxiv.org/abs/2409.01183}{\textit{arXiv, 2024}}).
  This classical functional from information theory is shown to be non-increasing along the flow of the non-linear PDE for all physically relevant particle interactions. The proof consists in establishing
  a new functional inequality on the sphere of Log-Sobolev type. This new a priori estimate on solutions yields global-in-time well posedness of the equation, in particular in the case of very singular interactions, a left open question up to this work. 
\end{abstract}

\begin{altabstract}
  Cette note est consacrée à un résultat récemment obtenu par L.~Silvestre, C.~Villani et l'auteur de ce texte  sur l'information de Fisher  pour l'équation de Boltzmann homogène en espace   (\href{https://arxiv.org/abs/2409.01183}{\textit{arXiv, 2024}}).
  Nous verrons que cette fonctionnelle de la théorie de l'information décroît le long du flot de l'équation non-linéaire pour toutes les interactions interparticulaires physiquement importantes.
  La démonstration consiste à établir une nouvelle inégalité fonctionnelle sur la sphère de type Log-Sobolev. Cette nouvelle estimée a priori sur les solutions permet de montrer le caractère bien posé de l'équation, notamment dans le cas des interactions très singulières, une question  restée ouverte jusqu'à ce travail. 
\end{altabstract}

\maketitle

\section{Introduction}

We report on the recent contribution \cite{imbert2024monotonicityfisherinformationboltzmann} about the space-homogeneous Boltzmann equation,
\begin{equation} \label{e:boltz}
  \partial_t f = Q(f,f), \quad t>0, x \in \R^d.
\end{equation}
The unknown function $f$ is real-valued, non-negative, and defined on $(0,+\infty) \times \R^d$. It represents the  density function associated
with the dynamics of a rarefied gas. The collision operator $Q(f,f)$ is defined by,
\[ Q (f,f) := \int_{w \in \R^d} \int_{\sigma' \in S^{d-1}} \bigg( f(v')f(w') - f(v) f(w) \bigg) B (|v-w|, \sigma \cdot \sigma') \dsigma' \dw \]
where $S^{d-1}$ denotes the unit sphere of $\R^d$. The function $B$ encodes the choice of particle interaction potential from which
the force created in any pair of particles derives. It is called the \emph{collision kernel}. Velocities $v'$ and $w'$ and the direction $\sigma$ are defined
by the following formulas,
\begin{equation}\label{e:vwcos}
  v' = \frac{v+w}2 + \frac{|v-w|}{2} \sigma', \quad w' = \frac{v+w}2 - \frac{|v-w|}{2} \sigma', \quad \sigma = \frac{v-w}{|v-w|}.
\end{equation}

\subsection{Entropy and Fisher information}

The Boltzmann equation describes the dynamics of the  density function of particles of a gas in the phase space (position $x$ and velocity $v$). 
After J.~Maxwell derived it in 1867 \cite{maxw:67}, L.~Boltzmann \cite{Boltzmann1970} introduced the concept of \emph{entropy} of the gas of particles
in order to study the long time behaviour of solutions,
\[ S (f) := - \int (\log f) f.\]
He made the seminal observation that this functional increases along the flow of the equation. This observation is now known
as Boltzmann's H-theorem and can be thought as a quantitative version of the irreversibility principle from Thermodynamics.  

The mathematical study of this non-linear equation lying at the center of  statistical physics started with works by T.~Carleman \cite{zbMATH03126493}.  
In order to better understand properties of solutions, he assumed that the gas is space-homogeneous: the repartition of velocities of particles
is the same at every position $x$. Even under such a strong assumption, the study of the resulting equation~\eqref{e:boltz} is very challenging.
Since the 1930's, an particularly in the last fifty years, numerous important contributions to the study of \eqref{e:boltz} were made. In many (if not most)
of these subsequent works, in particular since the mid 1990's,
the use of the \emph{entropy production} term, that is to say the opposite of the time derivative of the entropy, is key. 
We refer the reader to \cite{imbert2024monotonicityfisherinformationboltzmann} and references therein for more details. 
Despite these important contributions, the most singular interactions between particles were still out of reach, partly because the estimate of the
entropy production term were never proved sufficient to control solutions in these cases.  

In \cite{imbert2024monotonicityfisherinformationboltzmann}, this gap is closed by establishing a new a priori estimate: 
the Fisher information of a function $f \colon \R^d \to (0,+\infty)$, defined by
\[ I (f) := \int_{\R^d} |\nabla \log f|^2 f \dv \]
is shown to decrease with time for all physically relevant and important kernels: hard spheres, hard / moderately soft / very soft potentials,
under Grad's cut-off assumption of without it.

This work follows the breakthrough by N.~Guillen and L.~Silvestre \cite{landaufisher2023} about the Landau equation. This other nonlinear kinetic equation can be obtained from the Boltzmann equation through the so-called \emph{grazing collision} limit. The reason for considering with the Landau equation instead of the Boltzmann one is the following. 
Boltzmann's collision kernel does not make sense for Coulomb interactions because  the corresponding collision kernel turns out to be too singular. For this reason, L.~Landau \cite{landau:Coulomb:36} introduced an equation that nowadays bears his name.

\subsection{Boltzmann's collision operator}

\subsubsection{Collision kernels}

The collision kernel $B$, appearing in the formula defining Boltzmann's collision operator $Q$,  can take various forms, depending on the nature of
collisions between particles. 
\bigskip

Collision kernels $B(r,c)$ of the form $\alpha (r/2) b (c)$  are widely considered in the literature, starting with the seminal article
by J.~Maxwell himself. The functions $\alpha$ and $b$ are non-negative and respectively defined in $(0,+\infty)$ and $[-1,1]$. Such collision kernels cover the case of \emph{hard spheres} and \emph{inverse power-law potentials}.
The hard sphere case correspond to particles colliding like billiard balls. 
As far as inverse power-law potentials are concerned, they are inversely proportional to the $q$ power of the
distance between interacting particles. The parameter $q > d-1$ so that the potential decreases faster than the Coulombian one and $q \ge (d+1)/2$ so that
Condition~\eqref{e:levy} is satisfied. In contrast with the hard sphere case, the corresponding collision kernels can be very singular. 
\begin{example}[Hard spheres and inverse power laws]
  For hard spheres, $\alpha (r) = r$ and $b (\cos \theta ) =  |\sin (\theta /2)|^{3-d}$ or equivalently $b(c) = 2^{d-3}(1 -c)^{\frac{3-d}2}$. In particular, $b$ is constant in three dimensions.
  For inverse-power laws, $\alpha (r)=r^\gamma$ and $b(c) \simeq (1-c^2)^{-\frac{d-1+2s}2}$ with $2s = \frac{d-1}{q-1}$ and $\gamma :=1-4s$.
  In particular, $s \in [0,1]$ and $\gamma \in [-3,1]$.
\end{example}

Our approach allows us to deal with more general kernels, not necessarily factorized, see \cite{imbert2024monotonicityfisherinformationboltzmann}. In this note, we stick to the factorized case but  we consider general functions $b$. We just saw that they can be singular. This being said, they have to satisfy,
\begin{equation} \label{e:levy}
  \forall \sigma \in S^{d-1}, \qquad \int_{S^{d-1}} ( 1 - (\sigma \cdot \sigma')^2) b(\sigma \cdot \sigma') \dsigma' < +\infty.
\end{equation}
This condition  ensures that Boltzmann's collision operator makes sense. 
It is is reminiscent of the definition of Lévy measures \cite{sato}.

\subsubsection{Spherical linear Boltzmann operator}

It is useful and somewhat natural when studying \eqref{e:boltz} to consider the following integro-differential operator on the sphere,
\[ \cB f (\sigma) = \int_{\sigma' \in S^{d-1}} \bigg (f(\sigma') - f(\sigma) \bigg) b (\sigma' \cdot \sigma) \dsigma'.\]
D.~Bakry and M.~\'Emery \cite{zbMATH03894218} associated a  \emph{carré du champ} operator to any diffusion semi-group, such as the semi-group generated by the Laplacian operator $\Delta$ on the sphere.
They also considered the iterated \emph{carré du champ} $\Gamma^2$. In the case of the Laplacian on the sphere, it is given by $\Gamma^2_{\Delta,\Delta} (f,g)  = \nabla_\sigma f \cdot \nabla_\sigma g$. 
A  \emph{carré du champ} operator $\Gamma_\cB$ and an iterated one $\Gamma^2_{\cB,\cB}$ can also be defined for the integro-differential operator $\cB$. Notice that it is not a diffusion in the sense of Bakry and \'Emery. 
We may next intertwine these two diffusion operators $\Delta$ and $\cB$ and consider,
\[ \Gamma^2_{\cB,\Delta}(f,g) = \frac12 \left( \cB (\Gamma_\Delta (f,g)) - \Gamma_\Delta (\cB f,g) - \Gamma_\Delta (f, \cB g)  \right) .\]
We notice that it is not obvious that this quantity is non-negative and we will see that computing it in dimension $d \ge 3$ is delicate (see Section~\ref{s:curvature}).  

\subsection{A criteria for the monotonicity of the Fisher information}

We first give a sufficient condition for the monotonicity of the Fisher information. It relates the function $\alpha$ with the largest constant $\Lambda_b \ge 0$ such that
the following functional inequality holds: for all functions $f \colon S^{d-1} \to (0,\infty)$ such that $f (-\sigma) = f(\sigma)$,
\begin{equation}
  \label{e:log-sob}
  \int_{S^{d-1}} \Gamma^2_{\cB,\Delta} (\log f, \log f) f \dsigma \ge \Lambda_b \iint_{S^{d-1} \times S^{d-1}} \frac{(f(\sigma') - f(\sigma))^2}{f(\sigma')+f(\sigma)} b (\sigma' \cdot \sigma) \dsigma' \dsigma.
\end{equation}
\begin{theorem}[Sufficient condition for the monotonicity of Fisher -- {\cite{imbert2024monotonicityfisherinformationboltzmann}}] \label{t:criterion}
  Assume $b$ satisfies \eqref{e:levy} and the function $\alpha$ satisfies
  \[ \frac{r |\alpha' (r)|}{2 \alpha (r)} \le \sqrt{\Lambda_b}.\]
Then any solution of the space-homogeneous Boltzmann equation~\eqref{e:boltz} with collision kernel $B = \alpha (r/2) b (\cos \theta)$ is non-increasing in time. 
\end{theorem}

The proof of this theorem is very close to Guillen and Silvestre's proof of the monotonicity of Fisher for Landau \cite{landaufisher2023}.
It is rather quick and we will review it almost completely in Section~\ref{s:criterion}. 

\subsection{The log-Sobolev inequality on the sphere}

The criterion contained in Theorem~\ref{t:criterion} is useful if we can get a ``good'' lower bound on the constant $\Lambda_b$ in the cases of interest.
For large dimensions, it is possible to get a lower bound for any kernel. 
\begin{theorem}[Lower bound for $\Lambda_b$ in dimension $d \ge 3$] \label{t:curvature}
Assume that $b$ satisfies \eqref{e:levy} and $d \ge 3$. Then \eqref{e:log-sob} holds true with $\Lambda_b \ge d-2$. 
\end{theorem}

In dimension $2$, it is possible to construct a collision kernel $b$ such that \eqref{e:log-sob}
holds only for $\Lambda_b = 0$, see the counter-example constructed in \cite[Lemma~8.1]{imbert2024monotonicityfisherinformationboltzmann}. If we restrict ourselves to a smaller class of kernels $b$, then we can prove $\Lambda_b >d$ and in particular $\Lambda_b >0$ in the plane. 
\begin{theorem}[Lower bound for $\Lambda_b$ for subordinated kernels] \label{t:subordinate}
  Let $u_t(c)$ denote the \emph{heat kernel} on the sphere $S^{d-1}$: the solution $U$
  of the heat equation on $S^{d-1}$ with initial datum $f$ is given by
  $U (t,\sigma ) = \int_{S^{d-1}} f(\sigma') u_t (\sigma' \cdot \sigma) \dsigma'$. 
  Assume that $b$ can be written as,
  \[ b (c) = \int_0^{+\infty} u_t (c) \omega (t) \dd t \]
  for some measurable function $\omega \colon (0,\infty) \to [0,\infty)$. 
   Then \eqref{e:log-sob} holds true with $\Lambda_b >d$ .
\end{theorem}
\begin{remark}
  Invariance by rotations of the Laplacian on the sphere 
  is used to get the integral representation  of solutions $U$ of
  the heat equation  mentioned in the statement.
\end{remark}
\begin{remark}
  The proof of this theorem yields a more precise lower bound on $\Lambda_b$ than $\Lambda_b >d$
  as we shall briefly see in the dedicated section below. 
\end{remark}

\subsection{Physically relevant kernels}

In this subsection, we comment on how the previous results can be applied in dimension $2$ and $3$
to address the cases of hard spheres and inverse power laws   discussed above.
In order to do so, we will need an elementary observation related to comparable kernels. 

\paragraph{Comparable kernels.}
When two kernels $b$ and $b_0$ are comparable, then so are the corresponding optimal constants $\Lambda_b$ and $\Lambda_{b_0}$. Here is a precise statement.
\begin{lemma}[Comparing kernels]\label{l:compare}
  Let $b$ and $b_0$ be two collision kernels. Assume that there exist two constants $c_0$ and $C_0$ such that for all $c \in [-1,1]$, we have
  \[ c_0 [ b_0 (c) + b_0 (-c) ] \le b (c) + b (-c) \le C_0 [b_0(c) + b_0 (-c)].\]
  Then $\Lambda_b \ge \frac{c_0}{C_0} \Lambda_{b_0}$.
\end{lemma}

\paragraph{Constant kernels.} Constant kernels make part of subordinate kernels. Indeed, they correspond to constant weight functions $\omega$. In particular, $\Lambda_b >d$ for
constant kernels.

\paragraph{Hard spheres.}
We start with hard sphere interaction. In this case, we recall that $\gamma =1$ and $\bhs (c) = 2^{d-3} \sqrt{1-c}^{3-d}$.

In the planar case, we thus have $\bhs (c) = \frac12 \sqrt{1-c}$.
In particular, $\bhs(c) + \bhs(-c) \in [\frac1{\sqrt{2}},\frac12]$. We thus can compare $\bhs$ (in the sense of Lemma~\ref{l:compare}) to $b_0 \equiv 1/2$, for which we know that $\Lambda_{b_0} \ge 2$ (see the first paragraph of this subsection). We conclude that $\Lambda_\bhs \ge 2 \frac{\sqrt{2}}2 = \sqrt{2}$. In particular $2 \sqrt{\Lambda_\bhs} \ge 1 = \gamma$.   

In dimension $3$, the collision kernel $\bhs$ is constant thus $\Lambda_\bhs \ge 3$ and $2 \sqrt{\Lambda_{\bhs}} \ge 1$. 

\paragraph{Inverse power laws in dimension $3$.} We know from Theorem~\ref{t:curvature} that for any kernel satisfying Condition~\eqref{e:levy}, we have
$\Lambda_b  \ge 1$ in dimension $3$. This is true in particular for kernels corresponding to inverse power law potentials. This result tells us that
we can address readily the case $|\gamma| \le 2$ or equivalently $s \le 3/4$ (or $q \ge 7/3$). Recall that very soft potentials correspond to
$\gamma +2s \le 0$ \textit{i.e.}  $s \ge 1/2$. 

The remaining very soft potential cases are adressed thanks to Lemma~\ref{l:compare} and some numerical computations by L.~Silvestre \cite{numerics}.
More precisely, inverse power law collision kernels are numerically compared with subordinate ones associated with some explicit weight functions.

\subsection{Global well-posedness for very soft potentials}

An important consequence of the monotonicity of the Fisher information along the flow of the Boltzmann equation is global well-posedness for
very soft potentials. If global well-posedness has been known for a while for  hard ($\gamma \ge 0$) and moderately soft  ($-2s \le \gamma  \ge 0$) potentials -- see for instance \cite{he2012},
it was a well-known open problem in the case of very soft potentials. 
\begin{theorem}[{\cite[Theorem~1.6]{imbert2024monotonicityfisherinformationboltzmann}}]
Assume $d=3$ and consider a collision kernel $B(r,c)$  of the form $r^\gamma b (c)$ with $b \simeq (1-c^2)^{-\frac{d-1+2s}2}$ for some $\gamma \in (-3,0]$ and $s \in (0,1)$. 
There exists $q >1$ such that for all $\fin (v)$ satisfying $(1+ |v|^2)^q \fin (v) \in L^\infty$, the Boltzmann equation~\eqref{e:boltz} has a global
smooth solution with initial datum $\fin$.
\end{theorem}
Since this note focuses on the monotonicity of the Fisher information, the  reader that is interested in global well-posedness is referred to \cite{imbert2024monotonicityfisherinformationboltzmann} for further references and more details. 

\subsection{The Fisher information for kinetic equations in the literature}

The Fisher information was first considered in the study of kinetic equations by H.~P.~McKean \cite{mckean1966} in the study of Kac's model for Maxwell molecules. 
Then G.~Toscani \cite{toscani1992} proved that the Fisher information decreases along the flow of the space-homogeneous Boltzmann equation in space dimension $2$ and again for Maxwell molecules.
 C.~Villani \cite{villani1998boltzmann} extended this first result about Boltzmann to any space dimension (but still with Maxwell molecules). We already mentioned the
work by N.~Guillen and L.~Silvestre \cite{landaufisher2023} about the space-homogeneous Landau equation for a large class of potentials. Notice that C.~Villani recently wrote a  very complete review paper \cite{villani-crete}. 

For previous known results concerned with either Log-Sobolev inequality, or global wellposedness of the space-homogeneous Boltzmann equation, the reader is referred to
\cite{imbert2024monotonicityfisherinformationboltzmann}. 

\subsection{Organization of the note and notation}

\paragraph{Organization.} The remainder of this note is organized as follows. Section~\ref{s:criterion} is dedicated to the proof of Theorem~\ref{t:criterion}.
This result consists in a criterion ensuring that the Fisher information decreases along the flow of the space-homogeneous Boltzmann equation. The two other sections
contain estimates of the best constant $\Lambda_b$ appearing in the log-Sobolev inequality. In Section~\ref{s:curvature}, it is assumed that
dimension is larger than $3$ and that collision kernels satisfy Condition~\eqref{e:levy}. In Section~\ref{s:subord}, a specific class of collision kernels
are considered (related to subordinate Brownian motions on the sphere) and a lower bound on $\Lambda_b$ is derived in any dimension larger than $2$. 

\paragraph{Notation.} We work with the Euclidian space $\R^d$ with $d \ge 2$. For $v,w \in \R^d$, $v \cdot w$ denotes the scalar product and $|v|$ denotes the Euclidian norm. 
The unit sphere is denoted by $S^{d-1}$ and $\Delta$ denotes the Laplacian operator on this sphere.

For two functions $f(v)$ and $g(v)$, $f \otimes g$ denotes the function $f(v)g(w)$
defined on $\R^d \times \R^d$. At some point, we will use polar coordinates for a variable in $\R^d$. However, using polar coordinates or not,  the domain integration $\R^d \times \R^d$ or $\R^d \times (0,+\infty) \times S^{d-1}$ is simply written $\R^{2d}$.

\section{Linking the time derivative of Fisher with the log-Sobolev inequality}
\label{s:criterion}

In this section, we review the proof of the criterion contained in Theorem~\ref{t:criterion} for the monotonicity in time of the Fisher information of any solution of the
space-homogeneous Boltzmann equation. We will only skip a few computations, so that this section is almost self-contained. 

\subsection{Tensorization}

For collision kernels $B$ of the form $\alpha (r/2) b(\sigma \cdot \sigma')$, the collision operator $Q$ can be written
\begin{equation}
  \label{e:QcQ}
  Q (f,f) =  \int_{\R^d} \cQ (f \otimes f) \dw
\end{equation}
where
the linear Boltzmann operator $\cQ$ is defined for $F \colon \R^d \times \R^d \to \R$ by 
\[ \cQ F(v,w) =  \alpha (|v-w|/2) \int_{\sigma'} \bigg( F(v',w') - F(v,w) \bigg) b (\sigma' \cdot \sigma) \dsigma', \]
and velocities $v',w'$ and $\sigma$ are still given by \eqref{e:vwcos}. 

One can then consider the linear equation,
\begin{equation}
  \label{e:linear}
  \partial_t F = \cQ F, \quad t>0, (v,w) \in \R^d \times \R^d.
\end{equation}

A remarkable observation made by N.~Guillen and L.~Silvestre is that,
given a collision kernel $b$, 
if the Fisher information of solutions of the \emph{linear} equation \eqref{e:linear}
decreases along time, then so does the Fisher information of solutions
of the non-linear equation~\eqref{e:boltz}. One way to prove such a result
is to relate Gâteaux derivatives of  $f \mapsto I(f)$ and $F \mapsto I(F)$ at $F=f\otimes f$. 
\begin{lemma}[Gâteaux derivative of the Fisher information] \label{l:tensor}
  Let $f \colon \R^d \to \R$ be non-negative and smooth.
  Then for $F = f \otimes f$,
  \[ \langle I'(f), Q(f,f) \rangle = \frac12\langle I'(F), \cQ F \rangle .\]
\end{lemma}
\begin{proof}
  We start from the right hand side of the equality contained in the statement.
  \begin{align}
 \nonumber    \langle I'(F), \cQ F \rangle  = & 2 \int_{\R^d \times \R^d} \frac{\nabla_v F \cdot \nabla_v \cQ F}{F} +2 \int_{\R^d \times \R^d} \frac{\nabla_w F \cdot \nabla_w \cQ F}{F}  \\
 \nonumber                                    & - \int_{\R^d \times \R^d} \frac{|\nabla_v F|^2}{F} \cQ F - \int_{\R^d \times \R^d} \frac{|\nabla_w F|^2}{F^2} \cQ F .\\
    \intertext{We now use that $F(v,w) = F(w,v)$ in order to get,}
   \label{e:I'(F)} \langle I'(F), \cQ F \rangle      = & 4 \int_{\R^d \times \R^d} \nabla_v \log F \cdot \nabla_v \cQ F - 2 \int_{\R^d \times \R^d} |\nabla_v \log F|^2 \cQ F \\
   \nonumber  = & 4 \int_{\R^d} \nabla_v \log f  \cdot \left\{ \int_{\R^d} \nabla_v \cQ F \dw \right\}- 2 \int_{\R^d} |\nabla_v  \log f|^2 \left\{ \int_{\R^d} \cQ F \dw\right\} \\
         \nonumber    = & 4 \int_{\R^d} \nabla_v \log f  \cdot \nabla_v Q (f,f)- 2 \int_{\R^d} |\nabla_v \log f|^2 Q (f,f) 
  \end{align}
  where we used \eqref{e:QcQ} to get the last line. We now recognize $2 \langle I'(f), Q(f,f) \rangle$. 
\end{proof}

\subsection{Polar coordinates}

In order to study the Gâteaux derivative of $F$ at $F = f \otimes f$, it is better to consider
\[ z = \frac{v+w}2, \quad r = \frac{|v-w|}2, \quad \sigma = \frac{v-w}{|v-w|}.\]
With such a change of variables in hand, we notice that
\[ \cQ F (z,r,\sigma) = \alpha (r) \int_{S^{d-1}} \bigg( F(z,r,\sigma') - F(z,r,\sigma) \bigg) b (\sigma' \cdot \sigma) \dsigma'.\]
We write next,
\begin{equation}\label{e:fisher-dec}
  I (F) = I_z (F) + I_r (F) + I_\sigma (F) 
\end{equation}
with
\[ \left\{\begin{aligned}
  I_z (F) &= \int_{\R^{2d}} \frac{|\nabla_z F|^2}{F} r^{d-1} \dz \dr \dsigma,  \\
  I_r (F) &= \int_{\R^{2d}} \frac{|\nabla_r F|^2}{F} r^{d-1} \dz \dr \dsigma,  \\
  I_\sigma (F) &= \int_{\R^{2d}} \frac{|\nabla_\sigma F|^2}{r^2F} r^{d-1} \dz \dr \dsigma.  
\end{aligned}\right. \]
Lengthy but straightforward computations \cite[Lemmas~3.1 and 3.2]{imbert2024monotonicityfisherinformationboltzmann} yield,
\begin{align}
\label{e:I'zF}  \langle I_z (F), \cQ F \rangle &\le 0 ,\\
\label{e:I'rF}  \langle I_r (F), \cQ F \rangle &= \frac12 \iint_{S^{d-1} \times \R^{2d}} \frac{(\alpha'(r))^2}{\alpha(r)} \frac{(F(z,r,\sigma') - F(z,r,\sigma))^2}{F(z,r,\sigma')+ F(z,r,\sigma)} r^{d-1} \dsigma' \dz \dr \dsigma.
\end{align}
The next lemma contains the key computation.
\begin{lemma}[Gâteaux derivative of the Fisher information in $\sigma$] \label{l:gateaux-sigma}
For smooth functions $F$, we have
\[ \langle I_\sigma' (F), \cQ F \rangle = - 2  \int_{\R^{2d}}  \frac{\alpha(r)}{r^2} \Gamma^2_{\cB,\Delta} (\log F,\log F) F \; r^{d-1} \dz \dr \dsigma .\]
\end{lemma}
\begin{proof}
  Remember that $\Gamma_\Delta ( F, G) = \nabla_\sigma F \cdot \nabla_\sigma G$. In particular, we can compute the Gâteaux derivative of
  $I_\sigma(F)$ after writing $I_\sigma (F) = \int_{\R^{2d}} \Gamma_\Delta (\log F,\log F) F r^{d-3} \dz \dr \dsigma$. This yields,
\begin{align*}
  &\langle I_\sigma' (F), \cQ F \rangle \\
  & = 2 \int_{\R^{2d}} \Gamma_\Delta \left(\log F,\frac{ \cQ F}{F} \right) F r^{d-3} \dz \dr \dsigma + \int_{\R^{2d}} \Gamma_\Delta (\log F,\log F) \cQ F r^{d-3} \dz \dr \dsigma \\
  & = 2 \int_{\R^{2d}}  \nabla_\sigma  F \cdot \nabla_\sigma \left( \frac{ \cQ F}{F} \right)  r^{d-3} \dz \dr \dsigma
    + \int_{\R^{2d}}  \Gamma_\Delta (\log F,\log F)  \cQ F r^{d-3} \dz \dr \dsigma. \\
  \intertext{We now write $\nabla_\sigma F \cdot \nabla_\sigma (\cQ F/F)$ as $\nabla_\sigma \log F \cdot \nabla_\sigma  \cQ F - \Gamma_\Delta (\log F,\log F) \cQ F$,}
  & = 2 \int_{\R^{2d}}  \nabla_\sigma  \log F \cdot \nabla_\sigma \cQ F \; r^{d-3} \dz \dr \dsigma
    - \int_{\R^{2d}}  \Gamma_\Delta (\log F,\log F)  \cQ F \;r^{d-3} \dz \dr \dsigma.
    \intertext{We now use that $\cQ$ and $\Delta_\sigma$ commute \cite[Lemma~2.4]{imbert2024monotonicityfisherinformationboltzmann} for the first term and that
    $\cQ$ is self-adjoint in $L^2$ for the second term,}
  & = - 2 \int_{\R^{2d}}    \log F \cQ \Delta  F \; r^{d-3} \dz \dr \dsigma
    - \int_{\R^{2d}}  \cQ \Gamma_\Delta (\log F,\log F)  F \;r^{d-3} \dz \dr \dsigma.
    \intertext{We  use again that $\cQ$ is self-adjoint in $L^2$ and integrate by parts the first term  to finally get,}
  & =  2 \int_{\R^{2d}} \nabla_\sigma \cQ   \log F \cdot \nabla_\sigma  F \; r^{d-3} \dz \dr \dsigma
    - \int_{\R^{2d}}  \cQ \Gamma_\Delta (\log F,\log F)  F \;r^{d-3} \dz \dr \dsigma.
\end{align*}
We conclude by writing $\nabla_\sigma F = F \nabla_\sigma \log F$ and recognize  $\Gamma_\Delta(\cQ \log F, \log F) F$ in the first term. 
\end{proof}
\begin{proof}[Sketch of proof of Theorem~\ref{t:criterion}]
  For any solution $f$ of \eqref{e:boltz}, consider the solution $F$ of \eqref{e:linear} such that $F= f\otimes f$ at initial time.
  Thanks to \eqref{e:fisher-dec}, \eqref{e:I'zF}, \eqref{e:I'rF} and Lemma~\ref{l:gateaux-sigma}, we have,
  \begin{align*}
    \langle I'(F),\cQ F \rangle \le
    &  \frac12 \iint_{S^{d-1} \times \R^{2d}} \frac{(\alpha'(r))^2}{\alpha(r)} \frac{(F(z,r,\sigma') - F(z,r,\sigma))^2}{F(z,r,\sigma')+ F(z,r,\sigma)} r^{d-1} \dsigma' \dz \dr \dsigma \\
    & - 2  \int_{\R^{2d}}  \frac{\alpha(r)}{r^2}\Gamma^2_{\cB,\Delta} (\log F,\log F) F \; r^{d-1} \dz \dr \dsigma .
  \end{align*}
  The condition imposed on $\alpha$ implies that $\langle I'(F),\cQ F \rangle  \le 0$. Now we conclude that the Fisher information of $f$ decreases along time
  thanks to Lemma~\ref{l:tensor}. 
\end{proof}

\section{Log-Sobolev inequality from curvature}
\label{s:curvature}

In this section and the following one, we derive lower bounds for the constant $\Lambda_b$ appearing in the log-Sobolev inequality \eqref{e:log-sob}.
The lower bound that we will get in the present section will be obtained by using the curvature of the sphere $S^{d-1}$. This lower bound is positive
only for $d \ge 3$. To get a result that also applies in the planar case, we will consider collision kernels coming from subordinate Brownian motions
on the sphere. In both cases, we will obtain a lower bound on $\Lambda_b$ by establishing two intermediate functional inequalities. 
\begin{lemma}[Reduction] \label{l:reduction}
  The Log-Sobolev inequality \eqref{e:log-sob}  holds  with $\Lambda_b =  \frac{2C_K}{C_P}$ as soon as the two following ones hold
for all smooth $F \colon S^{d-1} \to (0,+\infty)$ such that $F(-\sigma)=  F(\sigma)$,
\begin{eqnarray}
\label{e:gamma2}
  C_K \int_{S^{d-1}} |\nabla_\sigma \log F|^2 F \dsigma \le \int_{S^{d-1}} \Gamma^2_{\cB,\Delta} (\log F,\log F) F \dsigma, \\
    \label{e:hardy}
  \iint_{S^{d-1} \times S^{d-1}} (F(\sigma')-F(\sigma))^2 b (\sigma' \cdot \sigma) \dsigma \dsigma'  \le C_P \int_{S^{d-1}} |\nabla_\sigma F|^2 \dsigma.
\end{eqnarray}
\end{lemma}
\begin{proof}
  Using first  the elementary inequality $\frac{(a-b)^2}{a+b} \le 2 (\sqrt{a}-\sqrt{b})^2$, second \eqref{e:hardy} for $F =\sqrt{f}$ and third \eqref{e:gamma2}  yields \eqref{e:log-sob} with $\Lambda_b = \frac{C_P}{2C_K}$. Indeed,
  \begin{align*}
    \iint_{S^{d-1} \times S^{d-1}} \frac{(f(\sigma') - f(\sigma))^2}{f(\sigma')+f(\sigma)}
    & b (\sigma' \cdot \sigma) \dsigma' \dsigma \\
    & \le 2\iint_{S^{d-1} \times S^{d-1}} \left(\sqrt{f}(\sigma') - \sqrt{f}(\sigma) \right)^2 b (\sigma' \cdot \sigma) \dsigma' \dsigma \\
    & \le 2C_P \int_{S^{d-1}} \left|\nabla_\sigma \sqrt{f} \right|^2 \dsigma \\
    & = \frac{C_P}2 \int_{S^{d-1}} \left|\nabla_\sigma \log f \right|^2 f \dsigma \\
    &\le \frac{C_P}{2C_K} \int_{S^{d-1}} \Gamma^2_{\cB,\Delta} (\log f,\log f) f \dsigma.  \qedhere
  \end{align*}
\end{proof}

\subsection{A $\Gamma^2$ criterion}

In order to compute   $\Gamma^2_{\cB,\Delta}  = \frac12 \left( \cB |\nabla_\sigma F|^2 - 2\nabla_\sigma F \cdot \nabla_\sigma \cB F \right),$
we will have to compute $\nabla_\sigma \cB F$. Such a computation is easy when $d=2$ since in this case the sphere is a circle. But it is involved
for larger dimensions. We shall see soon that the following computation  made by C.~Villani in \cite[Lemma~2]{villani1998boltzmann} is very useful.
\begin{lemma}[Gradient of $\cB F$]\label{e:grad}
  Let $\sigma \in S^{d-1}$ and $G \colon S^{d-1} \to \R$ smooth. We have
  \[ \int_{S^{d-1}} G (\sigma') \nabla_{\sigma} [ b( \sigma' \cdot \sigma)] \dsigma' = \int_{S^{d-1}} \bigg[ M_{\sigma',\sigma} \nabla_\sigma G (\sigma') \bigg] b(\sigma' \cdot \sigma) \dsigma'\]
  where $M_{\sigma',\sigma} \colon T_{\sigma'} S^{d-1} \to T_{\sigma} S^{d-1}$ is defined by $M_{\sigma',\sigma} (x) = (\sigma' \cdot \sigma) x - (\sigma \cdot x) \sigma'$. 
\end{lemma}
\begin{remark}
  The operator $M_{\sigma',\sigma}$ is the restriction of $P_{\sigma',\sigma} \colon \R^d \to \R^d$ that maps $\sigma'$ to $\sigma$, is a rotation in the plane generated by $\sigma$ and $\sigma'$,
  and equals $(\sigma' \cdot \sigma)$ times the identity on the orthogonal of this plane \cite[\S~4.3]{imbert2024monotonicityfisherinformationboltzmann}. 
\end{remark}
With this lemma in hand, we can obtain the following (non-integrated) $\Gamma^2$ criterion. 
\begin{proposition}[{\cite[Lemma~6.1]{imbert2024monotonicityfisherinformationboltzmann}}]
  Let $b$ satisfy \eqref{e:levy} and $d >2$. For all smooth functions $F$ such that $F(-\sigma)=F(\sigma)$,
  we have
  \( \Gamma^2_{\cB,\Delta} (F,F) \ge C_K |\nabla_\sigma F|^2 \)
  with
  \[ C_K = \frac{d-2}{2(d-1)} \int_{S^{d-1}} ( 1 - (e_1 \cdot \sigma')^2) b (e_1 \cdot \sigma') \dsigma'.\]
\end{proposition}
\begin{proof}
  We start with computing $\nabla_\sigma \cB F$.
  \begin{align*}
    \nabla_\sigma \cB F  &= \nabla_\sigma \int_{S^{d-1}} (F(\sigma') - F (\sigma)) b (\sigma'\cdot \sigma) \dsigma' \\
                         & =   \int_{S^{d-1}}  \bigg\{ -\nabla_\sigma F (\sigma)  b (\sigma'\cdot \sigma)  +  (F(\sigma') - F (\sigma)) \nabla_\sigma \bigg[ b (\sigma'\cdot \sigma) \bigg] \bigg\} \dsigma' \\
                         & =   \int_{S^{d-1}}  \bigg\{ -\nabla_\sigma F (\sigma)  b (\sigma'\cdot \sigma)  +  M_{\sigma',\sigma} \nabla_\sigma F(\sigma')  b (\sigma'\cdot \sigma) \bigg\} \dsigma'. 
  \end{align*}
  We now use  definitions of $\Gamma_\Delta$ and $\Gamma^2_{\cB,\Delta}$ in order to write,
  \begin{align*}
    \Gamma^2_{\cB,\Delta} & = \frac12 \left( \cB |\nabla_\sigma F|^2 - 2\nabla_\sigma F \cdot \nabla_\sigma \cB F \right) \\
                          & = \frac12 \int_{S^{d-1}} \bigg\{ |\nabla_\sigma F (\sigma')|^2 + |\nabla_\sigma F(\sigma)|^2 - 2 \nabla_\sigma F(\sigma) \cdot M_{\sigma',\sigma} \nabla_\sigma F(\sigma')  \bigg\} b(\sigma' \cdot \sigma) \dsigma'. \\
    \intertext{We now use that $x \cdot M_{\sigma',\sigma} y = M_{\sigma,\sigma'} x \cdot y$ for $x \in T_\sigma S^{d-1}$ and $y \in T_{\sigma'} S^{d-1}$, }
                          & = \frac12 \int_{S^{d-1}} \bigg\{ |\nabla_\sigma F (\sigma')|^2 + |\nabla_\sigma F(\sigma)|^2 - 2 M_{\sigma,\sigma'} \nabla_\sigma F(\sigma) \cdot  \nabla_\sigma F(\sigma')  \bigg\} b(\sigma' \cdot \sigma) \dsigma' \\
                          & = \frac12 \int_{S^{d-1}} \bigg\{ |\nabla_\sigma F (\sigma') - M_{\sigma,\sigma'} \nabla_\sigma F(\sigma)  |^2 + |\nabla_\sigma F(\sigma)|^2 - |M_{\sigma,\sigma'} \nabla_\sigma F(\sigma)|^2  \bigg\} b(\sigma' \cdot \sigma) \dsigma' \\
                              & \ge \frac12 \int_{S^{d-1}} \bigg\{ |x|^2 - |M_{\sigma,\sigma'} x |^2  \bigg\} b(\sigma' \cdot \sigma) \dsigma'  \qquad \text{   with $x=\nabla_\sigma F(\sigma)$} \\
                          & =  \frac12 |x|^2 \int_{S^{d-1}} \bigg\{ 1 - |M_{\sigma,\sigma'} e |^2  \bigg\} b(\sigma' \cdot \sigma) \dsigma'  \qquad \text{   with $e=\frac{x}{|x|} \in S^{d-1}$}.
  \end{align*}
Computing the integral in the last line yields the result. 
\end{proof}

\subsection{A Hardy-type inequality}

\begin{proposition} \label{p:hardy}
  Let $b$ satisfy \eqref{e:levy}.  For all smooth functions  $F$ such that $F(-\sigma)=F(\sigma)$,
  the Hardy-type inequality \eqref{e:hardy} holds with $C_P$ given for any $\sigma \in S^{d-1}$ by the formula,
  \[ C_P = \frac1{d-1} \int_{S^{d-1}}  (1 -(\sigma' \cdot \sigma)^2) b(\sigma' \cdot \sigma) \dsigma'.\]
\end{proposition}
The proof of this proposition relies on the spectral properties of the operators $\Delta$ and $\cB$. We recall that
$\cB$ is defined by,
\[ \cB f (\sigma) = \int_{S^{d-1}} (f(\sigma') - f(\sigma)) b (\sigma' \cdot \sigma) \dsigma' \]
and that $\cQ F (z,r,\sigma) = \alpha (r) \cB F(z,r,\cdot) (\sigma).$ Let us keep in minde that we already used above that
this integro-differential operator commutes with the Laplacian. 

Spectral properties of the Laplacian on the sphere are well-known, see for instance \cite{efthimiou2014}.
\begin{itemize}
\item Its eigenvalues are $\lambda_\ell = \ell (\ell+d-2)$ for any $\ell \ge 0$.
\item The eigenspace associated with $\lambda_\ell$ is finite dimensional and composed
  of $\ell$-spherical harmonics, that is to say of restrictions to the sphere of
  $\ell$-homogeneous harmonic polynomials.
\item Any $\ell$-spherical harmonic is a linear combination of rotations of $Y_\ell (\sigma)= a P_\ell (k \cdot \sigma)$ where $a \in \R$, $k \in S^{d-1}$ and $P_\ell$ is the Legendre polynomial of degree $\ell$.  
\end{itemize}
We already mentioned that $\cB$ and $\Delta$ commute. And both operators  commute with rotations.

Using these facts about $\Delta$ and $\cB$, one can prove  \cite[Lemmas~7.3 and 7.5]{imbert2024monotonicityfisherinformationboltzmann} that
\[ \cB Y_\ell = \tilde \lambda_\ell Y_\ell \]
with
\[ \tilde \lambda_\ell = \int_{S^{d-1}} (1-P_\ell (e_1 \cdot \sigma')) b (e_1 \cdot \sigma') \dsigma'\]
where $P_\ell$ is the $\ell$-Legendre polynomial with the normalization condition $P_\ell (1)=1$.
With those spectral properties, it is now easy to get the following estimate for $C_P$.
\begin{lemma}[{\cite[Lemma~7.4]{imbert2024monotonicityfisherinformationboltzmann}}]
\label{l:hardy}  The Hardy-type inequality~\eqref{e:hardy} holds true for all smooth functions $F$ such that $F(-\sigma)=F(\sigma)$ with
  \[ C_P = 2 \sup_{\ell \ge 1} \frac{\tilde{\lambda}_{2 \ell}}{\lambda_{2\ell}}.\]
\end{lemma}
\begin{proof}
  Decompose $F$ in spherical harmonics. Since it is even, it can be written
  \[ F = \sum_{\ell \ge 1} F_{2 \ell} Y_{2 \ell}.\]
  The announced formula then derives from the two elementary computations,
  \begin{align*}
    \int_{S^{d-1}} |\nabla_\sigma F|^2 &= -\int_{S^{d-1}} F \Delta_\sigma F = \sum_{\ell \ge 1} \lambda_{2 \ell} F_{2 \ell}^2, \\  
    \int_{S^{d-1}\times S^{d-1}} (F(\sigma') - F(\sigma))^2 b (\sigma' \cdot \sigma) \dsigma' \dsigma
      &= -2 \int_{S^{d-1}} F \cB F = 2 \sum_{\ell \ge 1} \tilde{\lambda}_{2 \ell} F_{2 \ell}^2. \qedhere
  \end{align*}
\end{proof}
In order to compute the supremum in the formula for $C_P$, we use an inequality about Legendre polynomials.
\begin{proposition}[{\cite[Proposition~7.6]{imbert2024monotonicityfisherinformationboltzmann}}] \label{p:legendre}
For all $\ell \ge 1$, we have $\frac{1-P_{2 \ell}}{\lambda_{2\ell}} \le \frac{1-P_2}{\lambda_2}$. 
\end{proposition}
The proof of this proposition relies on known representations of Legendre polynomials and 

\begin{proof}[Proof of Proposition~\ref{p:hardy}]
  Combining  Lemma~\ref{l:hardy} with formulas for $\lambda_\ell$, $\tilde{\lambda}_\ell$ and Proposition~\ref{p:legendre}, we get the result thanks to the fact that 
  $P_2 (x) = x^2 + (1-x^2)/(d-1)$ -- recall that we normalize $P_2$ so that $P_2 (1)=1$.
\end{proof}

\section{Lob-Sobolev inequality through subordination}
\label{s:subord}

In this last section, we will see how to get the log-Sobolev inequality~\eqref{e:log-sob} in any dimension $d \ge 2$ for collision kernels
$b$ that can be written as,
\[ b_\omega (c) = \int_0^\infty \omega (t) u_t (c) \dd t \]
for an arbitrary measurable function $\omega \colon (0,\infty) \to [0,\infty)$.

Following the reasoning that we went through for $d \ge 3$, the log-Sobolev inequality is established by proving  both  a $\Gamma^2$ criterion and a Hardy-type inequality.
\begin{proposition}[$\Gamma^2$ criterion for intertwined diffusions -- {\cite[Proposition~9.1]{imbert2024monotonicityfisherinformationboltzmann}}] \label{p:gamma2sub}
The $\Gamma^2$ criterion \eqref{e:gamma2} holds for all smooth functions $F$ such that $F(-\sigma) = F(\sigma)$ with
  \[ C_K = \int_0^\infty \omega (t) \frac{1-e^{-2 \Lambdaloc t}}{2} \dt \]
\end{proposition}

\begin{proposition}[A Hardy-type inequality -- {\cite[Proposition~9.2]{imbert2024monotonicityfisherinformationboltzmann}}] \label{p:hardysub}
  The Hardy-type inequality~\eqref{e:hardy} holds for all smooth functions $F$ such that $F(-\sigma) = F(\sigma)$ with
  \[ C_P = \int_0^\infty \omega (t) \frac{1-e^{-2dt}}{d} \dt .\]
\end{proposition}

The $\Gamma^2$ criterion for mixed diffusions ($\cB$ and $\Delta$) relies in the log-Sobolev inequality proved by N.~Guillen and L.~Silvestre \cite[Proposition~5.7.3]{landaufisher2023}. The constant appearing in this work (and denoted by $\Lambdaloc$ in the next statement)  was later improved by S.~Ji \cite{sehyun2024}.
\begin{theorem}[The $\Gamma^2$ criterion for the Laplacian] \label{t:local}
  For all smooth functions $F$ defined on the sphere and such that $F(-\sigma)=F(\sigma)$,
  \[ \Lambdaloc \int_{S^{d-1}} |\nabla_\sigma \log F|^2 F \dsigma \le \int_{S^{d-1}} \Gamma^2_{\Delta,\Delta} (\log F,\log F) F \dsigma  \]
  with $\Lambdaloc = d+3 -\frac1{d-1}$. In particular $\Lambdaloc > d$ for $d \ge 2$. 
\end{theorem}

In order to reduce the proof of the monotonicity of the Fisher information along the flow of the Boltzmann equation, we first observated (see Lemma~\ref{l:gateaux-sigma}) that the Gâteaux derivative of the Fisher information with respect to the $\sigma$-variable can be written in terms of $\Gamma^2_{\cB,\Delta}$. Through minor changes to the proof of this lemma, we can prove the following one.
\begin{lemma} \label{l:gateauxsub}
  For $F$ regular enough, we have \( \langle I'(F), \cB F \rangle = -2 \int_{S^{d-1}} \Gamma^2_{\cB,\Delta} (\log F,\log F) F \dsigma. \)
\end{lemma}
We now recall the definition of the spherical linear Boltzmann operator, that we shall denote $\cB_\omega$ in this section to emphasize the dependence on the weight $\omega$.
\begin{align*}
  \cB_\omega F & = \int_{S^{d-1}} ( F(\sigma') - F(\sigma) ) b_\omega (\sigma' \cdot \sigma) \dsigma' \\
  & = \int_0^\infty \int_{S^{d-1}} ( F(\sigma') - F(\sigma) ) u_t (\sigma' \cdot \sigma)  \omega (t) \dsigma' \dt \\
  & = \int_0^\infty (F_t - F) \omega (t) \dt
\end{align*}
where $F_t$ is the solution of the heat equation on the sphere with initial data $F$.
It is convenient to write
\begin{equation}\label{e:cbt}
  \cB_\omega F = \int_0^\infty \cB_t (F) \omega (t) \dt \quad \text{ with } \quad \cB_t F = F_t -F. 
\end{equation}

\begin{proof}[Proof of Proposition~\ref{p:gamma2sub}]
  Thanks to \eqref{e:cbt} and Lemma~\ref{l:gateauxsub} for $\cB_t$, we have
  \begin{align}
\nonumber  2 \int_{S^{d-1}}  \Gamma^2_{\cB,\Delta} (\log F,\log F) F \dsigma & = 2 \int_0^\infty  \int_{S^{d-1}}  \Gamma^2_{\cB_t,\Delta} (\log F,\log F) F \dsigma \omega (t) \dt \\
\nonumber                                                                      & = - \int_0^\infty \langle I'(F) , \cB_t F \rangle \omega (t) \dt \\
\nonumber                                                                      & = - \int_0^\infty \langle I'(F) ,  F_t-F \rangle \omega (t) \dt. \\
    \intertext{We now use the fact that $I$ is convex,}
    & \ge \int_0^\infty (I(F) - I (F_t)) \omega (t) \dt. \label{here}
  \end{align}
  We know that the Fisher information decreases along the flow of the heat equation. We can compute an explicit rate of convergence thanks to
  the log-Sobolev inequality proved by N.~Guillen and L.~Silvestre (recall Theorem~\ref{t:local}). Indeed, 
  \begin{align*}
    \frac{\dd}{\dt} I (F_t) &= \langle I'(F_t), \Delta F_t \rangle \\
                            & = -2 \int_{S^{d-1}} \Gamma^2_{\Delta,\Delta} (\log F_t,\log F_t) F_t \dsigma \\
    \intertext{(we used a well-known fact that can be recovered by adapting computations contained in the proof of Lemma~\ref{l:gateaux-sigma})}
    & \le - 2 \Lambdaloc I (F_t).
  \end{align*}
  In particular $I (F_t) \le \exp (-2 \Lambdaloc t) I (F)$. We conclude the proof by combining this inequality with \eqref{here}.
\end{proof}

The proof of the Hardy-type inequality in the case of subordinate kernels also relies on the spectral properties that
were used in the previous section.
\begin{proof}[Proof of Proposition~\ref{p:hardysub}]
  Let $F$ be decomposed as $\sum_{\ell \ge 1} F_{2 \ell} Y_{2 \ell}$.
  We remark that $F_t = \sum_{\ell \ge 1} e^{- \lambda_{2\ell} t} F_{2\ell} Y_{2\ell}$.
  In particular,
  \[ \cB_t F = \sum_{\ell \ge 1} \left( e^{- \lambda_{2\ell} t} -1 \right) F_{2\ell} Y_{2\ell}.\]
With such an observation in hand, we can write
  \begin{align*}
    \int_{S^{d-1} \times S^{d-1}} (F(\sigma') - F(\sigma))^2 b (\sigma' \cdot \sigma) \dsigma' \dsigma
    & = - 2 \int_{S^{d-1}} F \cB F \\
    &  = - 2 \int_0^\infty \left\{ \int_{S^{d-1}} F \cB_t F \right\} \omega (t) \dt \\
& = 2 \sum_{\ell \ge 1} \left( \int_0^\infty(1- e^{- \lambda_{2\ell} t}) \omega (t) \dt \right) F_{2\ell}^2.
  \end{align*}
  We conclude after observing that   \( \frac{1-e^{-\lambda_{2\ell} t}}{\lambda_{2\ell}} \le \frac{1-e^{-2dt}}{2d} \)
  (since $\lambda_{2\ell} \ge 2d$ for $\ell \ge 1$). 
\end{proof}

We finally check that we have all we need in order to prove Theorem~\ref{t:curvature}.
\begin{proof}[Proof of Theorem~\ref{t:curvature}]
  Combining Lemma~\ref{l:reduction} with Propositions~\ref{p:gamma2sub}, \ref{p:hardysub}, we get
  \[ \Lambda_b \ge d \frac{\int_0^\infty (1 - \exp (-2 \Lambdaloc t)) \dt}{\int_0^\infty (1 - \exp (-2 d t)) \dt}.\]
  Now the fact that $\Lambdaloc > d$ (Theorem~\ref{t:local}) implies that $\Lambda_b > d$. 
\end{proof}

\bibliographystyle{plain}
\bibliography{biblio}
\end{document}